\documentclass[journal,twoside,web]{ieeecolor}

\usepackage{lcsys}
\usepackage{multirow}
\usepackage{float}
\usepackage{makecell}
\usepackage{graphicx}
\usepackage{mathtools}
\usepackage{url}
\usepackage{textcomp,gensymb}
\usepackage{algorithm,algpseudocode} 
\usepackage{booktabs}
\usepackage{amsmath}
\usepackage{amssymb}
\usepackage{hyperref}

\usepackage{amsthm}

\newtheorem{thm}{Theorem}

\newtheorem{lemma}{Lemma}
\newtheorem{prop}{Proposition}

\global\long\def\norm#1{\lVert#1\rVert}%
\global\long\def\abs#1{\lvert#1\rvert}%
\global\long\def\vecc#1{\text{vec}(#1)}%

\usepackage{cite}
\usepackage{graphicx}
\usepackage{textcomp}

\begin{document}

\def\BibTeX{{\rm B\kern-.05em{\sc i\kern-.025em b}\kern-.08em
    T\kern-.1667em\lower.7ex\hbox{E}\kern-.125emX}}
\markboth{\journalname, VOL. XX, NO. XX, XXXX 2017}
{Author \MakeLowercase{\textit{et al.}}: Preparation of Papers for IEEE Control Systems Letters (August 2022)}

\title{Optimization-based Constrained  Funnel Synthesis for Systems with  Lipschitz Nonlinearities   via Numerical Optimal Control
}

\author{Taewan Kim, Purnanand Elango, Taylor P. Reynolds, \\ Beh\c{c}et A\c{c}\i kme\c{s}e,  \IEEEmembership{Fellow, IEEE}, and Mehran Mesbahi,  \IEEEmembership{Fellow, IEEE}
\thanks{Manuscript received 17 March 2023; revised 19 May 2023; accepted 13 June 2023. This work was supported by Air Force Office of Scientific Research grant FA9550-20-1-0053. (Corresponding author: Taewan Kim.)}
\thanks{Taewan Kim, Purnanand Elango, Beh\c{c}et A\c{c}\i kme\c{s}e, and Mehran Mesbahi are with the Department of Aeronautics and Astronautics, University of Washington, Seattle, WA 98195, USA (e-mail: twankim@uw.edu, pelango@uw.edu, behcet@uw.edu, mesbahi@uw.edu).}%
\thanks{Taylor P. Reynolds is with Amazon Prime Air, Seattle, WA 98117, USA (e-mail:tayreyno@amazon.com).}%
}
\maketitle
\thispagestyle{empty}

\begin{abstract}
This paper presents a funnel synthesis algorithm for computing controlled invariant sets and feedback control gains around a given nominal trajectory for dynamical  systems with locally Lipschitz nonlinearities and  bounded disturbances. The resulting funnel synthesis problem involves a differential linear matrix inequality (DLMI) whose solution satisfies a Lyapunov condition that implies invariance and attractivity properties. Due to these properties, the proposed method can balance  maximization of initial invariant funnel size, i.e., size of the {\em funnel entry}, and minimization of the size of the attractive funnel for attenuating the effect of disturbance. To solve the resulting funnel synthesis problem with the DLMI as constraints, we employ a numerical optimal control approach that uses a multiple shooting method to convert the problem into a finite dimensional semidefinite programming problem. This framework does not require piecewise linear system matrices and funnel parameters, which is typically assumed in recent related work. We illustrate the proposed funnel synthesis method with a numerical example.
\end{abstract}


\section{INTRODUCTION}
\IEEEPARstart{F}{unnel}, also referred to as tube, represents regions of finite-time controlled invariant state space for closed-loop systems equipped with an associated feedback control law around a given nominal trajectory \cite{rakovic2006simple}. Funnel synthesis refers to a procedure for computing both the controlled invariant  set and the corresponding feedback control law. Once we compute a library of funnels along different nominal trajectories, the resulting funnel can be used for different purposes such as real-time motion planning  \cite{majumdar2017funnel} and feasible trajectory generation \cite{reynolds2021funnel}. 

The studies in funnel synthesis  can be separated into two categories depending on whether they aim to maximize \cite{reynolds2021funnel,tobenkin2011invariant,fejlek2022computing} or minimize the size of the funnel \cite{majumdar2017funnel,kim2022joint}. The funnel computation inherently aims to maximize the size of the funnel to have a larger controlled invariant set in the state space. On the other hand, when it comes to systems under uncertainty or disturbances, the funnel size should be minimized to bound the effect of the uncertainty. For example, \cite{majumdar2017funnel} minimize the size of the funnel to prohibit collision with obstacles instead of imposing obstacle avoidance constraints directly. However, minimizing the size of the funnel is against the original purpose of having a large controlled invariant set in the state space. In this work, we provide a funnel synthesis algorithm that balances maximizing the size of the funnel and minimizing the effect of the bounded disturbance. To this end, we exploit invariance and attractivity conditions derived from Lyapunov theory \cite{corless1993control} by solving linear matrix inequalities (LMIs) \cite{kothare1996robust,accikmecse2011robust}  and imposing state and input constraints directly on the funnel.

When employing the Lyapunov condition, the resulting optimization problem has a differential inequality of the Lyapunov function in continuous-time for a finite-time interval. Since it is intractable to satisfy the inequality for all time in the given interval, many approaches focus on imposing the differential inequality at a finite number of node points \cite{majumdar2017funnel,tobenkin2011invariant,fejlek2022computing}. When a quadratic Lyapunov function with a time-varying positive definite (PD) matrix is employed, the resulting differential inequality ends up with a differential linear matrix inequality (DLMI). To solve the resulting DLMI, one can assume that first-order approximations (Jacobians matrices) of the nonlinear dynamics computed around the nominal trajectory are continuous piecewise linear in time. By applying the same piecewise linear parametrization to the PD matrix in the Lyapunov function, one can obtain a finite number of LMIs whose feasibility is a sufficient condition for the original DLMI \cite{reynolds2021funnel,malikov2020numerical}. The main downside of this approach is that the  assumption of piecewise linear system matrices may have large errors, and applying the same parametrization on the PD matrix can be conservative.


 In this paper, we provide a constrained funnel synthesis algorithm for locally Lipschitz nonlinear systems under bounded disturbance. To this end, we express the closed-loop system around the given nominal trajectory as a linear time-varying system having uncertain terms. Then, the DLMI is derived based on the Lyapunov condition that guarantees the invariance and the attractivity conditions. With the Lyapunov condition, the continuous-time funnel optimization problem maximizes the size of the funnel entry and minimizes the attractive funnel for attenuating the effect of disturbance. Furthermore, the proposed method can satisfy linear state and control constraints in a way that the resulting funnel around the given nominal trajectory remains inside the feasible sets of states and controls. 
 To convert the funnel synthesis problem into a finite dimensional semidefinite programming (SDP) problem, we employ a numerical optimal control approach with a multiple shooting method \cite{diehl2011numerical}. 

The contributions of this work are as follows: First, the proposed funnel synthesis approach for locally Lipschitz nonlinear systems provides a new optimization framework that can 1) balance maximizing the size of the funnel entry and minimizing the effect of the disturbance, and 2) guarantee the satisfaction of linear state and control constraints on the funnel. 
Second, we provide a new approach based on multiple shooting in numerical optimal control for solving the DLMI. 
This is in contrast to the prior approaches that assume the piecewise linear approximation on the PD matrix in the Lyapunov function.

The notations $\mathbb{R}$, $\mathbb{R}_{+}$ $\mathbb{R}_{++}$, and $\mathbb{R}^{n}$ are the field of real, nonnegative, positive numbers, and the $n$-dimensional Euclidean space, respectively. The set $\mathcal{N}_{q}^{r}$ is a finite set of consecutive nonnegative integers, i.e., $\{q,q+1,\ldots,r\}$. The symmetric matrix $Q=Q^{\top}(\succeq)\!\succ 0$ implies $Q$ is PD (PSD) matrix, and $(\mathbb{S}^n_{+})\, \mathbb{S}^n_{++}$ denotes the set of all PD (PSD) matrices whose size is $n\times n$. The symbols $\otimes$ is the Kronecker product. The notation {*} denotes the symmetric part of a matrix.
The squared root of a PSD matrix $A$ is defined as $A^\frac{1}{2}$ such that $A = A^\frac{1}{2} A^\frac{1}{2} $. We omit the time argument $t$ if it is clear from the context. The operation $\oplus$ is Minkowski sum. 




\section{Constrained Funnel Synthesis} \label{sec:funnel}
\subsection{Locally Lipschitz Nonlinear Systems}
Consider the following continuous-time dynamics:
\begin{align}
\dot{x}(t) & =f(t,x(t),u(t),w(t)),\quad t\in[t_{0},t_{f}], \label{eq:nonlinear_sys}
\end{align}
where  $x(t)\in\mathbb{R}^{n_x}$ is state and $u(t)\in\mathbb{R}^{n_u}$ is input. The vector-valued function $w(t)\in\mathbb{R}^{n_w}$ represents bounded disturbance such that $\norm{w(\cdot)}_\infty \leq 1$ where $\norm{w(\cdot)}_\infty \coloneqq \sup_{t\in[t_{0},t_{f}]}\norm{w(t)}$, and $t_{0}$ and $t_{f}$ are initial and final time, respectively. The function $f:\mathbb{R}_+ \times \mathbb{R}^{n_x} \times \mathbb{R}^{n_u} \times \mathbb{R}^{n_w} \rightarrow \mathbb{R}^{n_x}$ is assumed to be continuously differentiable. Suppose that a nominal trajectory $(\bar{x}(\cdot),\bar{u}(\cdot),\bar{w}(\cdot))$ is a solution of the system \eqref{eq:nonlinear_sys}.  Particularly, we choose a zero disturbance for the nominal trajectory, that is $\bar{w}(t)=0$ for all $t\in[t_{0},t_{f}]$. Then, we can convert \eqref{eq:nonlinear_sys} into the linear time-varying (LTV) system with the nonlinear remainder term via linearization around the nominal trajectory, resulting in the following Lur'e type system \cite{reynolds2021funnel,boyd1994linear}:
\begin{align}
\nonumber \dot{x}(t) & =A(t)x(t)+B(t)u(t)+F(t)w(t)+Ep(t),\\
p(t) & =\phi(t,q(t)), \quad q(t) =Cx(t)+Du(t)+Gw(t), \label{eq:Lure}
\end{align}
where $p(t)\in\mathbb{R}^{n_p}$ is a lumped nonlinearity represented by a nonlinear function $\phi(t)$ and its argument $q(t)\in\mathbb{R}^{n_q}$. The matrices $A(t)$, $B(t)$, and $F(t)$ are first-order approximations of the nonlinear dynamics \eqref{eq:nonlinear_sys} around the nominal trajectory.
The matrices $E\in\mathbb{R}^{n_x\times n_p}$, $C\in\mathbb{R}^{n_q\times n_x}$, $D\in\mathbb{R}^{n_q\times n_u}$, and
$G\in\mathbb{R}^{n_q\times n_w}$ are assumed to be time-invariant \footnote{ The matrices $E, C, D$, and $G$ are selector matrices with 0s and 1s to organize the nonlinearity of the system. The simplest case has $E=I$, $q=[x^\top, u^\top ,w^\top]^\top$, and $p = \phi(t,q) =  f(t,x,u,w) - Ax - Bu - Fw$.}. The more details in choosing these matrices could be found in \cite{reynolds2021funnel}.

With the state difference $\eta\coloneqq x-\bar{x}$, the difference dynamics can be derived as
\begin{align*}
\dot{\eta}(t) & =f(t,x,u,w)-f(t,\bar{x},\bar{u},0),\\
 & =A(t)\eta(t)+B(t)\xi(t)+F(t)w(t)+E\delta p(t),\\
\delta p(t) & =\phi(t,q(t))-\phi(t,\bar{q}(t)),\\
\delta q(t) & =C\eta(t)+D\xi(t)+Gw(t),
\end{align*}
where $\xi\coloneqq u-\bar{u}$ and $\delta q\coloneqq q-\bar{q}$ with $\bar{q}=C\bar{x}+D\bar{u}$. Since continuously differentiable functions are locally Lipschitz, $f$  and $\phi$ are locally Lipschitz. It follows that for all $t\in[t_{0},t_{f}]$
\begin{align*}
\norm{p-\bar{p}}_{2} & \leq\gamma(t)\norm{q-\bar{q}}_{2},\quad\forall\, q,\bar{q}\in\mathcal{Q},
\end{align*}
where $\gamma(t)\in\mathbb{R}_{+}$ is a Lipschitz constant for each $t$ and $\mathcal{Q} \subseteq \mathbb{R}^{n_x}$ is any compact set. By employing the linear feedback controller, that is $\xi(t)=K(t)\eta(t)$, the closed-loop system can be written as
\begin{subequations}\label{eq:diff_closedloop}
\begin{align}
\dot{\eta} & =(A+BK)\eta+Fw+E\delta p,\quad \norm{w(\cdot)}_\infty \leq 1, \\
\delta q & =C\eta+D\xi+Gw, \quad \norm{\delta p}_{2} \leq\gamma\norm{\delta q}_{2}.
\end{align}
\end{subequations}
With \eqref{eq:diff_closedloop}, we can express the nonlinear closed-loop system as the LTV system  having the state and input dependent uncertainty $\delta p$. This could be a conservative way to handle the nonlinear system, but it allows us to design a quadratic Lyapunov function with which we can guarantee the invariance and attractivity conditions of the funnel for the original nonlinear system.

\subsection{Lyapunov Conditions}
With a continuously differentiable positive definite matrix-valued function $Q:\mathbb{R}_+ \rightarrow \mathbb{S}_{++}^{n_{x}}$, the Lyapunov function is defined as 
\begin{align}
V(t,\eta) & \coloneqq\eta^{\top}(t)Q^{-1}(t)\eta(t). \label{eq:Lyapunov_function}
\end{align}
Here we aim to impose the following Lyapunov condition for the closed-loop system \eqref{eq:diff_closedloop}: 
\begin{subequations}\label{eq:Lyapunov_condition}
\begin{align}
\dot{V}(t,\eta) & \leq-\alpha V(t,\eta),\\
\text{for all }\norm{\delta p(t)}_{2} & \leq\gamma(t)\norm{\delta q(t)}_{2}, \\
\text{and }V(t,\eta) & \geq\norm{w(t)}_{2}^{2}, \quad\forall\, t  \in[t_{0},t_{f}], 
\end{align}
\end{subequations}
where $\alpha\in\mathbb{R}_{++}$ is a decay rate. With the above Lyapunov condition, we can establish the following lemma.

\begin{lemma} \label{lemma1}
Suppose that the Lyapunov condition \eqref{eq:Lyapunov_condition} holds with a positive definite
matrix-valued continuous function $Q(t)$ , then the time-varying ellipsoid defined as
\begin{align}
\mathcal{E}(t) = \{\eta\mid\eta^{\top}Q(t)^{-1}\eta\leq 1\}, \label{eq:minimial_funnel}
\end{align}
is invariant for the closed-loop system \eqref{eq:diff_closedloop}, that is, if $\eta(\cdot)$ is any solution with $\eta(t_{0})\in\mathcal{E}(t_{0})$, then $\eta(t)\in\mathcal{E}(t)$ for all $t\in[t_{0},t_{f}]$. 
Furthermore, the ellipsoid $\mathcal{E}(t)$ is attractive such that for any solution $\eta(\cdot)$, the following holds for all $t \in [t_0,t_f]$: 
\begin{align}
V(t,\eta(t))\leq\max\{e^{-\alpha(t-t_{0})}V(t_{0},\eta(t_{0})),1\}.\label{eq:attractive_condition}
\end{align}
\end{lemma}
\noindent The above lemma can be deduced from \cite[Lemma B10]{accikmecse2011robust} and \cite[Lemma 1]{accikmecse2003robust}, so here we skip the proof.

Now we define a \textit{invariant state funnel} with a pair of $Q$ in \eqref{eq:Lyapunov_function} and a continuous scalar-valued function $c:\mathbb{R}_{+}\rightarrow(0,1]$ as
\begin{align} \label{eq:state_funnel}
\mathcal{E}_{c}(t)\coloneqq\left\{\,\eta\,\middle|\,\eta^{\top}Q(t)^{-1}\eta\leq\frac{1}{c(t)}\right\},
\end{align}
where the function $c(t)$ satisfy the following condition:
\begin{align}
\frac{1}{c(t)}\geq\max\left\{1,e^{-\alpha(t-t_{0})}\frac{1}{c(t_{0})}\right\},\label{eq:condition_c}
\end{align}
with $0<c(t_{0})\leq1$. With the ellipsoid $\mathcal{E}_{c}(t)$ having $1/c(t)$ as the support value, we show the invariance property of $\mathcal{E}_{c}(t)$ in the following lemma.

\begin{lemma} \label{lemma2}
The ellipsoid $\mathcal{E}_{c}(t)$ defined in \eqref{eq:state_funnel} with \eqref{eq:condition_c} is invariant for the closed-loop system \eqref{eq:diff_closedloop} such that if $\eta(\cdot)$ is any solution with $\eta(t_{0})\in\mathcal{E}_{c}(t_{0})$, then $\eta(t)\in\mathcal{E}_{c}(t)$ for all $t\in[t_{0},t_{f}]$.
\end{lemma}
\begin{proof}
If the solution $\eta(\cdot)$ satisfies $\eta(t_{0})\in\mathcal{E}(t_{0})$, it is trivial to prove the invariance of $\mathcal{E}_{c}(t)$ since $\mathcal{E}(t)$ is invariant and $\mathcal{E}(t) \subseteq  \mathcal{E}_{c}(t)$. 
Consider the solution $\eta(\cdot)$ such that $\eta(t_{0})\in\mathcal{E}_{c}(t_{0})\setminus\mathcal{E}(t_{0})$. By the attractivity condition \eqref{eq:attractive_condition}, we have $V(t,\eta(t))\leq\max\{e^{-\alpha(t-t_{0})}V(t_{0},\eta(t_{0})),1\}$. It follows from $V(t_{0},\eta(t_{0}))\leq1/c(t_{0})$ and $0<c(t_{0})\leq1$ that $V(t,\eta(t))\leq\max\{e^{-\alpha(t-t_{0})}\frac{1}{c(t_{0})},1\}\leq1/c(t)$ for $t\in[t_{0},t_{f}]$. This completes the proof.
\end{proof}

\begin{figure}
\begin{center}
\vspace*{0.05in}
\includegraphics[width=8.5cm]{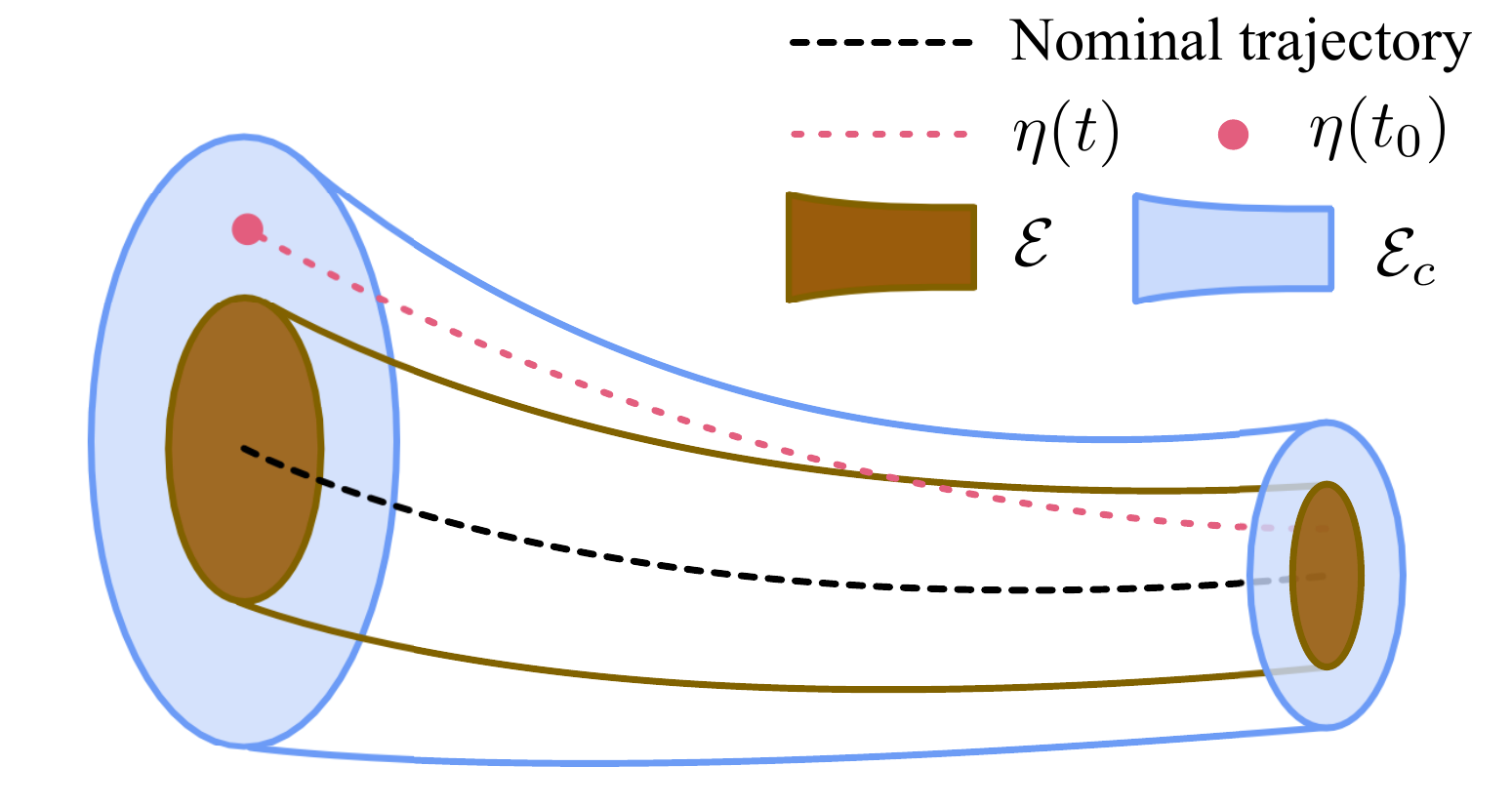}  
\caption{
Illustration of the ellipsoids $\mathcal{E}(t)$ and $\mathcal{E}_{c}(t)$. An example of solution $\eta(t)$ is given as a dashed red line. Since $\mathcal{E}$ is attractive, any solution $\eta(\cdot)$ starting with $\eta(t_{0})\in\mathcal{E}_{c}(t_{0})\backslash\mathcal{\mathcal{E}}(t_{0})$ converges to $\mathcal{E}$ if $t_{f}$ is sufficiently large. The proposed funnel synthesis aims to maximize the size of $\mathcal{E}_c (t_0)$ and minimize that of $\mathcal{E}(t)$ for all $t$ in $[t_0,t_f]$.
} 
\label{fig:funnel}
\end{center}
\end{figure}

The illustration of both the ellipsoids $\mathcal{E}(t)$ and $\mathcal{E}_{c}(t)$ is given in Figure \ref{fig:funnel}. Any solution $\eta(\cdot)$ of the closed-loop system \eqref{eq:diff_closedloop} starting at $\mathcal{E}_{c}(t_{0})$ remains in the state funnel $\mathcal{E}_{c}(t)$ for all $t\in[t_{0},t_{f}]$ because of the invariance condition of $\mathcal{E}_{c}$ derived in Lemma \ref{lemma2}. Furthermore, the solution $\eta(\cdot)$ starting at $\mathcal{E}_{c}(t_{0})$ converges to the ellipsoid $\mathcal{E}$ if $t_f$ is sufficiently large because of the attractivity of $\mathcal{E}$ given in Lemma \ref{lemma1}. Since we use the attractivity condition of $\mathcal{E}$ as a key property for our funnel generation, we refer to $\mathcal{E}$ in \eqref{eq:minimial_funnel} as an \textit{attractive funnel}.



Additionally, with the linear feedback control $\xi=K\eta$, the condition $\eta\in\mathcal{E}_{c}$ implies that $\xi$ is in the following ellipsoid \cite{boyd2004convex,kurzhanski1997ellipsoidal}:
\begin{align}
\mathcal{E}_{u}=\{(KQK^{\top})^{\frac{1}{2}}y\mid\norm y_{2}\leq1/\sqrt{c},y\in\mathbb{R}^{n_{u}}\}.\label{eq:input_funnel}
\end{align}
The set $\mathcal{E}_{u}$ represents the ellipsoid inside which the input deviation $\xi$ remains, so we refer to $\mathcal{E}_{u}$ as an \textit{invariant input funnel}. Now, we are ready to derive the DLMI condition that guarantees the invariant and attractive conditions.

\begin{thm}  \label{theorem}
Suppose that there exists $Q:[t_{0},t_{f}]\rightarrow\mathbb{S}_{++}^{n_{x}}$, $Y:[t_{0},t_{f}]\rightarrow\mathbb{R}^{n_{u}\times n_{x}}$, $\nu:[t_{0},t_{f}]\rightarrow\mathbb{R}_{++}$, $0<\lambda_{w}$, and $0<\alpha$ such that the following differential matrix inequality holds for all $t\in[t_{0},t_{f}]$:
\begin{align}
H\coloneqq \left[\begin{array}{cccc}
M-\dot{Q} & * & * & *\\
\nu E^{\top} & -\nu I & * & *\\
F^{\top} & 0 & -\lambda_{w}I & *\\
CQ+DY & 0 & G & -\nu\frac{1}{\gamma^{2}}I
\end{array}\right] & \preceq0,\label{eq:DLMI_nonlinear}\\
M\coloneqq QA^{\top}+Y^{\top}B^{\top}+AQ+BY+\alpha Q & +\lambda_{w}Q.\nonumber
\end{align}
Then, the Lyapunov condition \eqref{eq:Lyapunov_condition} holds for the closed-loop system \eqref{eq:diff_closedloop} with $K=YQ^{-1}$. Thus, with $Q(t)$ and $K(t)$ satisfying the DLMI \eqref{eq:DLMI_nonlinear}, the ellipsoid $\mathcal{E}(t)$ in \eqref{eq:minimial_funnel} is invariant and attractive, and $\mathcal{E}_{c}(t)$ in \eqref{eq:state_funnel} is invariant by Lemma \ref{lemma1} and Lemma \ref{lemma2}. 
\end{thm}

\begin{proof}
By definition of positive definiteness and S-procedure \cite{boyd1994linear}, the sufficient condition for the Lyapunov condition \eqref{eq:Lyapunov_condition} is that if there exists scalars $\lambda_{p}>0$ and $\lambda_{w}>0$ such that
\begin{equation*}\resizebox{0.91\hsize}{!}{%
$\begin{aligned}
\left[\begin{array}{ccc}
 \bar{M}-Q^{-1}\dot{Q}Q^{-1} &  * &  *\\
 E^{\top}Q^{-1} &  0 &  *\\
 F^{\top}Q^{-1} &  0 &  0
\end{array}\right] +\lambda_{p}
 {\underbrace{
 \left[\begin{array}{ccc}
 C_{k}^{cl} &  0 &  G_{k}\\
 0 &  I &  0
\end{array}\right]
}_{\coloneqq C_G}}^\top
\left[\begin{array}{cc}
 \gamma^{2}I &  0\\
 0 &  -I
\end{array}\right]
C_G
\\
+ \lambda_{w}\left[\begin{array}{ccc}
 Q^{-1} &  0 &  0\\
 0 &  0 &  0\\
 0 &  0 &  -I
\end{array}\right]   \preceq0,
\end{aligned}$%
}\end{equation*}
where $\bar{M}\coloneqq A_{cl}^{\top}Q^{-1}+Q^{-1}A_{cl} + \alpha Q^{-1}$. Applying Schur complement, and then multiplying either side by $\text{diag}\{Q,\lambda_{p}^{-1}I,I,I\}$ complete the proof with $\nu\coloneqq\lambda_{p}^{-1}$.
\end{proof}
\noindent Notice that the above differential matrix inequality \eqref{eq:DLMI_nonlinear} is linear in $\dot{Q}$, $Q$, $Y$, and $\nu$ once $\lambda_{w}$ and $\alpha$ are fixed.

\subsection{Feasibility of State and Input Funnels}
The funnel synthesis of the proposed work aims to be not only invariant but also feasible, so constraints on the invariant state and input funnels should be satisfied. The feasible sets for the state and input funnels can be described as
\begin{align*}
\mathcal{X} & =\{x\mid h_{i}(x)\leq0,\quad i=1,\ldots,m_{x}\},\\
\mathcal{U} & =\{u\mid g_{j}(u)\leq0,\quad j=1,\ldots,m_{u}\},
\end{align*}
where $h_{i}:\mathbb{R}^{n_{x}}\rightarrow\mathbb{R}$ and $g_{j}:\mathbb{R}^{n_{u}}\rightarrow\mathbb{R}$ are assumed to be at least once differentiable (possibly nonconvex) functions. 
Since it is not tractable to impose the nonconvex constraints on the ellipsoid funnel, we linearize the above constraints around the nominal trajectory, resulting in the following polyhedral constraints sets \cite{reynolds2021funnel}:
\begin{align*}
\mathcal{P}_{x} & =\{x\mid(a_{i}^{h})^{\top}x\leq b_{i}^{h},\quad i=1,\ldots,m_{x}\},\\
\mathcal{P}_{u} & =\{u\mid(a_{j}^{g})^{\top}u\leq b_{j}^{g},\quad j=1,\ldots,m_{u}\},
\end{align*}
where $(a_{i}^{h},b_{i}^{h})$ and $(a_{j}^{g},b_{j}^{g})$ are first-order approximations of $h_{i}$ and $g_{j}$, respectively. The inclusions $\mathcal{P}_{x}\subseteq\mathcal{X}$ and $\mathcal{P}_{u}\subseteq\mathcal{U}$ hold if the function $h_{i}$ is a concave function, such as ellipsoidal obstacle avoidance constraints. Here we assume that the input constraint set $\mathcal{P}_u$ is bounded in $\mathbb{R}^{n_u}$ to prohibit the control input from being arbitrarily large. This assumption usually holds because the unbounded input is not allowable in practice.

Now we aim to design $\mathcal{E}_{c}$ in \eqref{eq:state_funnel} and $\mathcal{E}_{u}$ in \eqref{eq:input_funnel} with $Q$ and $K$ such that $\{\bar{x}\}\oplus\mathcal{E}_{c} \subseteq\mathcal{P}_{x}$, $\{\bar{u}\}\oplus\mathcal{E}_{u} \subseteq\mathcal{P}_{u}$. These conditions could be equivalently written as \cite{boyd1994linear}
\begin{align*}
\norm{(Q/c)^{\frac{1}{2}}a_{i}^{h}}_{2}  \leq b_{i}^{h}-(a_{i}^{h})^{\top}\bar{x},\quad i&=1,\ldots,m_{x},\\
\norm{(K(Q/c)K^{\top})^{\frac{1}{2}}a_{j}^{g}}_{2}  \leq b_{j}^{g} - (a_{j}^{g})^{\top}\bar{u},\quad j&=1,\ldots,m_{u}.
\end{align*}
Squaring both sides and applying Schur complement equivalently generates
\begin{align}
0 & \preceq\left[\begin{array}{cc}
\left(b_{i}^{h}-(a_{i}^{h})^{\top}\bar{x}\right)^{2}c & (a_{i}^{h})^{\top}Q\\
Qa_{i}^{h} & Q
\end{array}\right] ,\label{eq:const_state}\\
0 & \preceq\left[\begin{array}{cc}
\left(b_{j}^{g}-(a_{j}^{g})^{\top}\bar{u}\right)^{2}c & (a_{j}^{g})^{\top}Y^\top\\
Ya_{j}^{g} & Q
\end{array}\right], \label{eq:const_input} \\
i&=1,\ldots,m_{x}, \quad j=1,\ldots,m_{u}. \nonumber
\end{align}
The feasibility conditions \eqref{eq:const_state} and \eqref{eq:const_input} are linear in $Q$, $Y$, and $c$.

\subsection{Objectives}

The goal of the  funnel synthesis aims to 1) maximize the size of invariant funnel entry $\mathcal{E}_c(t_0)$ from which the system can remain inside the invariant funnel and converge to the attractive funnel, and 2) minimize the size of the attractive funnel $\mathcal{E}(t)$ for all $t$ in $[t_0,t_f]$ for the disturbance attenuation.

First, the volume of funnel entry $\mathcal{E}_c(t_0)$ is proportional to $\det{Q(t_0)/c(t_0)}$ \cite{boyd2004convex}. Since a log function is increasing, it is equivalent to maximizing $\log\det(Q(t_0)/c(t_0))$. It follows that
\begin{align*}
\log\det (Q(t_0)/c(t_0)) = -n_x \log c(t_0) + \log\det Q(t_0).
\end{align*}
Hence, to maximize the volume of the funnel entry, we minimize $c(t_0)$ and $-\log\det Q(t_0)$, both of which are convex functions. Second, minimizing the volume of the set $\mathcal{E}(t)$ is equivalent to minimizing $\log\det Q(t)$ that is a concave function. Since minimizing the concave function is a nonconvex problem,  we instead minimize the maximum radius of $\mathcal{E}(t)$ that is equal to the squared root of the maximum eigenvalue of $Q(t)$ \cite{boyd2004convex}. Therefore, instead of minimizing the volume, we minimize the maximum eigenvalue of $Q(t)$ that is a convex function.

In summary, the funnel synthesis aims to minimize a cost function $J$ given as
\begin{align}
J=w_{c}c(t_{0})-w_{Q_0}\log\det Q(t_0)+\int_{t_{0}}^{t_{f}}\bar{w}_{Q}v^{Q}(t)\text{d}t, \label{eq:cost_function} \\
\text{ with }Q(t)\preceq v^{Q}(t)I,\quad\forall\, t\in[t_{0},t_{f}], \label{eq:eq_with_slack_variable}
\end{align}
where $v^{Q}(t)\in\mathbb{R}_{++}$ is a slack variable introduced to minimize the maximum eigenvalue of $Q(t)$, and $w_{c}, w_{Q_0}, \bar{w}_{Q}\in\mathbb{R}_{++}$ are user-defined weight parameters.

\subsection{Continuous-time funnel synthesis problem}
The continuous-time funnel synthesis problem can be formulated as follows:
\begin{subequations} \label{eq:contin_funnel_problem}
\begin{align}
    \underset{\resizebox{0.3\hsize}{!}{$Q(t), Y(t), c(t), \nu(t), v^Q(t)$}}{\operatorname{minimize}}~~& \eqref{eq:cost_function} \\[-0.1cm]
    \operatorname{subject~to}~~~~~&  \forall\, t\in[t_0,t_f], \\
    & \eqref{eq:condition_c},\eqref{eq:DLMI_nonlinear},\eqref{eq:const_state}, \eqref{eq:const_input}, \eqref{eq:eq_with_slack_variable}, \\
    & Q(t_0) \succeq c(t_0) Q_i , \label{eq:boundary_condition1} \\
    & Q(t_f) \preceq c(t_f) Q_f , \label{eq:boundary_condition2}
\end{align}
\end{subequations}
where the matrices $Q_i \in \mathbb{S}_{++}^{n_x}$ and $Q_f \in \mathbb{S}_{++}^{n_x}$ are constant parameters used for the boundary conditions. These boundary conditions imply $\mathcal{E}_c(t_0) \supseteq \{\eta \mid \eta^\top Q_i^{-1} \eta \leq 1\}$ and $\mathcal{E}_c(t_f) \subseteq \{\eta \mid \eta^\top Q_f^{-1} \eta \leq 1\}$.

\section{Optimizing Funnel via Optimal Control} \label{sec:OCP}

The problem formulated in \eqref{eq:contin_funnel_problem} is an infinite-dimensional continuous-time optimization problem, so it is not readily straightforward to solve it numerically. Here we discuss a way to transform the problem into a finite-dimensional discrete-time convex problem.

\subsection{Changing a DLMI to a Differential Matrix Equality}

In this subsection, we illustrate how the funnel synthesis problem \eqref{eq:contin_funnel_problem} can be interpreted as a continuous-time optimal control problem. Observe that the DLMI in \eqref{eq:DLMI_nonlinear} can be equivalently converted into a differential matrix equality (DME) by introducing a PSD-valued slack variable $Z(t)\in\mathbb{S}_{+}^{n_{z}}$ with $n_{z}=n_{x}+n_{p}+n_{w}+n_{q}$ as follows:

\begin{align}
H +\underbrace{\left[\begin{array}{cccc}
Z^{11} & * & * & *\\
Z^{21} & Z^{22} & * & *\\
Z^{31} & Z^{32} & Z^{33} & *\\
Z^{41} & Z^{42} & Z^{43} & Z^{44}
\end{array}\right]}_{\coloneqq Z} & =0, \quad Z \succeq 0, \label{eq:DME}
\end{align}
where $H$ is defined in \eqref{eq:DLMI_nonlinear} and $Z^{ij}(t)$ have appropriate sizes for all $i,j\in\{1,\ldots,4\}$. The first-row and first-column block has the following form:
\begin{align}
\dot{Q}(t)=M(t)+Z^{11}(t). \label{eq:semi-DME}
\end{align}
with $M(t)$ defined in \eqref{eq:DLMI_nonlinear}. The DME \eqref{eq:semi-DME} can be interpreted as a differential equation for a linear dynamical system where $Q$ is a state, and $Y$ and $Z^{11}$ are control inputs.

To derive further, we define the following vectors using the vectorization operation:
\begin{align}
q\coloneqq\vecc Q,y\coloneqq\vecc Y,z^{11}\coloneqq\vecc{Z^{11}}, \label{eq:def_vectors}
\end{align}
where the operation $\vecc{\cdot}$ stacks the columns to make a single vector. Then the DME \eqref{eq:semi-DME} can be equivalently expressed with the vector variables in \eqref{eq:def_vectors} as
\begin{align}
\dot{q}(t) & =A_{q}(t)q(t)+B_{q}(t)y(t)+S_{q}(t)z^{11}(t), \label{eq:ode_q}
\end{align}
with
\begin{align*}
A_{q} & =(I\otimes A)+(A\otimes I)+(\alpha+\lambda_{w})(I\otimes I),\\
B_{q} & =(I\otimes B)+(B\otimes I)K^{c}, \quad S_{q} =(I\otimes I),
\end{align*}
where $K^{c}\in\mathbb{R}^{n_{x}n_{u}\times n_{x}n_{u}}$ is a commutation matrix \cite{magnus1979commutation} such that $K^{c}\vecc N=\vecc{N^{\top}}$ for any arbitrary matrix $N\in\mathbb{R}^{n_{u}\times n_{x}}$. 

\subsection{Multiple Shooting Numerical Optimal Control}
To transform \eqref{eq:contin_funnel_problem} into the finite-dimensional discrete-time optimal control problem, we first choose uniform time grids as $t_{k}=t_0 + \frac{k}{N}(t_{f}-t_{0})$ for all $k\in\mathcal{N}_0^{N}$.
The decision variables and the Lipschitz constant $\gamma$ at each node point are set as $\triangle_k = \triangle (t_k)$ where a placeholder $\triangle$ represents $Q,Y,Z,c,\nu$, and $\gamma$.

We apply continuous piecewise linear interpolation for $Y$, $Z$, $\nu$ and $c^{-1}$ for each $k\in\mathcal{N}_0^{N-1}$ as follows:
\begin{align}
\nonumber \square(t) & =\lambda_{k}^{m}(t)\square_{k}+\lambda_{k}^{p}(t)\square_{k+1},\quad\forall\, t\in[t_{k},t_{k+1}],\\
\lambda_{k}^{m}(t) & =\frac{t_{k+1}-t}{t_{k+1}-t_{k}},\quad\lambda_{k}^{p}(t)=\frac{t-t_{k}}{t_{k+1}-t_{k}}, \label{eq:FOH}
\end{align}
where a placeholder $\square$ stands for $Y,\nu,Z,c^{-1}$. Notice that we apply the piecewise linear interpolation to the inverse of $c$, that is $c^{-1}$,
not $c$ itself. With this interpolation and additional conditions, we can show that $c(t)$ satisfies the condition \eqref{eq:condition_c} for all $t\in[t_{0},t_{f}]$.
\begin{prop}
Suppose that for each subinterval $c(t)$ satisfies
\begin{align*}
c(t) & =\frac{c_{k}c_{k+1}}{\lambda_{k}^{m}(t)c_{k+1}+\lambda_{k}^{p}(t)c_{k}},
\forall\, t\in[t_{k},t_{k+1}],\forall\, k\in\mathcal{N}_{0}^{N-1},
\end{align*}
and
\begin{align}
0 & <c_{k}\leq1, \quad e^{-\alpha(t_{k}-t_{0})}c_{k}\leq c_{0},\quad\forall\, k\in\mathcal{N}_{0}^{N}. \label{eq:condition_c_discrete}
\end{align}
Then, $c(t)$ satisfies the condition \eqref{eq:condition_c}.
\end{prop}
\begin{proof}
We want to show that $\frac{1}{c(t)}\geq\max\{1,e^{-\alpha(t-t_{0})}\frac{1}{c(t_{0})}\}$ for all $t\in[t_{0},t_{f}]$. The condition \eqref{eq:condition_c_discrete} implies $1/c_{k}\geq\max\{1,e^{-\alpha(t_{k}-t_{0})}\frac{1}{c(t_{0})}\}$ for all $k\in\mathcal{N}_0^N$. Notice that  $\max\{1,e^{-\alpha(t-t_{0})}\frac{1}{c(t_{0})}\}$ is convex in $t$,
and $1/c(t)$ is the convex combination of two points
$1/c_{k}$ and $1/c_{k+1}$ for $t\in[t_{k},t_{k+1}]$. Thus, it follows
from the definition of the convex function that $\frac{1}{c(t)}\geq\max\{1,e^{-\alpha(t-t_{0})}\frac{1}{c(t_{0})}\}$ for $t\in[t_{k},t_{k+1}]$.
Since this holds for all $k\in\mathcal{N}_0^{N-1}$, we complete the proof.
\end{proof}

The PD-valued function $Q(t)$ is
not assumed to be piecewise linear, so $q(t)$ is not. Instead, $q(t)$ is the solution of the ordinary differential equation in \eqref{eq:ode_q}. 
Since the system \eqref{eq:ode_q} is linear, we can equivalently express it through discretization with the interpolation \eqref{eq:FOH} as
\begin{align}
q_{k+1}&=A_{k}^{q}q_{k}+B_{k}^{-}y_{k}+B_{k}^{+}y_{k+1}+S_{k}^{-}z_{k}^{11}+S_{k}^{+}z_{k+1}^{11}, \nonumber \\
&\forall\, k\in \mathcal{N}_0^{N-1}, \label{eq:discrete_dynamics}
\end{align}
where $q_k = \vecc{Q_k}$, $y_k = \vecc{Y_k}$, and $z^{11}_k = \vecc{Z^{11}_k}$. The matrices $A_k^q$, $B_k^-$, $B_k^+$, $S_k^-$ and $S_k^+$ are corresponding discretized matrices. More details in obtaining these matrices could be found in \cite{malyuta2019discretization}. The other blocks in \eqref{eq:DME} are imposed as
\begin{equation}\resizebox{0.89\hsize}{!}{%
$\begin{aligned}
\nu_{k}E+Z_{k}^{21}=0,-v_{k}I+Z_{k}^{22}=0,F_k^{\top}+Z_{k}^{31}=0,Z_{k}^{32}=0, -\lambda_{w}I+Z_{k}^{33}=0,\\
 CQ_{k}+DY_{k}+Z_{k}^{41}=0, Z_{k}^{42}=0, G+Z_{k}^{43}=0,-\nu_{k}/\gamma_{k}^{2}I+Z_{k}^{44}=0,
 \end{aligned}$%
 }\label{eq:other_blocks}\end{equation}
for all $k$ in $\mathcal{N}_0^{N}$ where $F_k=F(t_k)$.

\subsection{Discrete-time convex funnel synthesis problem}
The discrete-time funnel synthesis problem can be formulated as follows:
\begin{subequations} \label{eq:discrete_funnel_problem}
\begin{align}
    \underset{\resizebox{0.33\hsize}{!}{$Q_k, Y_k, c_k, \nu_k, v^Q_k, \forall\, k \in \mathcal{N}_0^{N}$}}{\operatorname{minimize}}~&  \begin{array}{l}
    w_c c_0 -w_{Q_0}\log\det Q_0\\[0.1cm] 
~~~~~~+ \sum_{k=0}^N w_Q v^Q_k
    \end{array}\\
    \operatorname{subject~to}~~~~&  Z \succeq 0,\eqref{eq:condition_c_discrete},\eqref{eq:discrete_dynamics},\eqref{eq:other_blocks}, \\
    & \eqref{eq:const_state},\eqref{eq:const_input}, \eqref{eq:eq_with_slack_variable}
    ,\label{eq:discrete_constraints} \\
    & Q_0 \succeq c_0 Q_i, \quad  Q_N \preceq c_N Q_f, 
\end{align}
\end{subequations}
where the constraints  \eqref{eq:const_state}, \eqref{eq:const_input}, \eqref{eq:eq_with_slack_variable}
in \eqref{eq:discrete_constraints} are imposed at each $t=t_k$ for all $k$ in  $\mathcal{N}_0^{N}$. The continuous cost function \eqref{eq:cost_function} is discretized by lower sum between the subintervals, resulting in $w_Q = \bar{w}_Q (t_f-t_0)/N_{node}$. The optimization problem \eqref{eq:discrete_funnel_problem} is convex with LMI constraints, resulting in a SDP problem, so that we can solve it using any SDP solver.

\subsection{Inter-sample constraint violation}

One of the key issues in direct shooting approaches for optimal control is the inter-sample constraint violation 
\cite{richards2015inter} 
that the constraint violation can occur during subintervals since the constraints are enforced only at temporal nodes, not for all time. This is also an issue for the proposed method as well as other funnel generation approaches \cite{majumdar2017funnel,reynolds2021funnel,tobenkin2011invariant}.
Future research will explore how to impose relevant constraints for all time by exploiting the form of the solutions of DME \eqref{eq:DME} \cite{dieci1994positive}.


\section{Numerical Simulation}\label{sec:sim}

\begin{figure}
\begin{center}
\vspace*{0.05in}
\includegraphics[width=8.5cm]{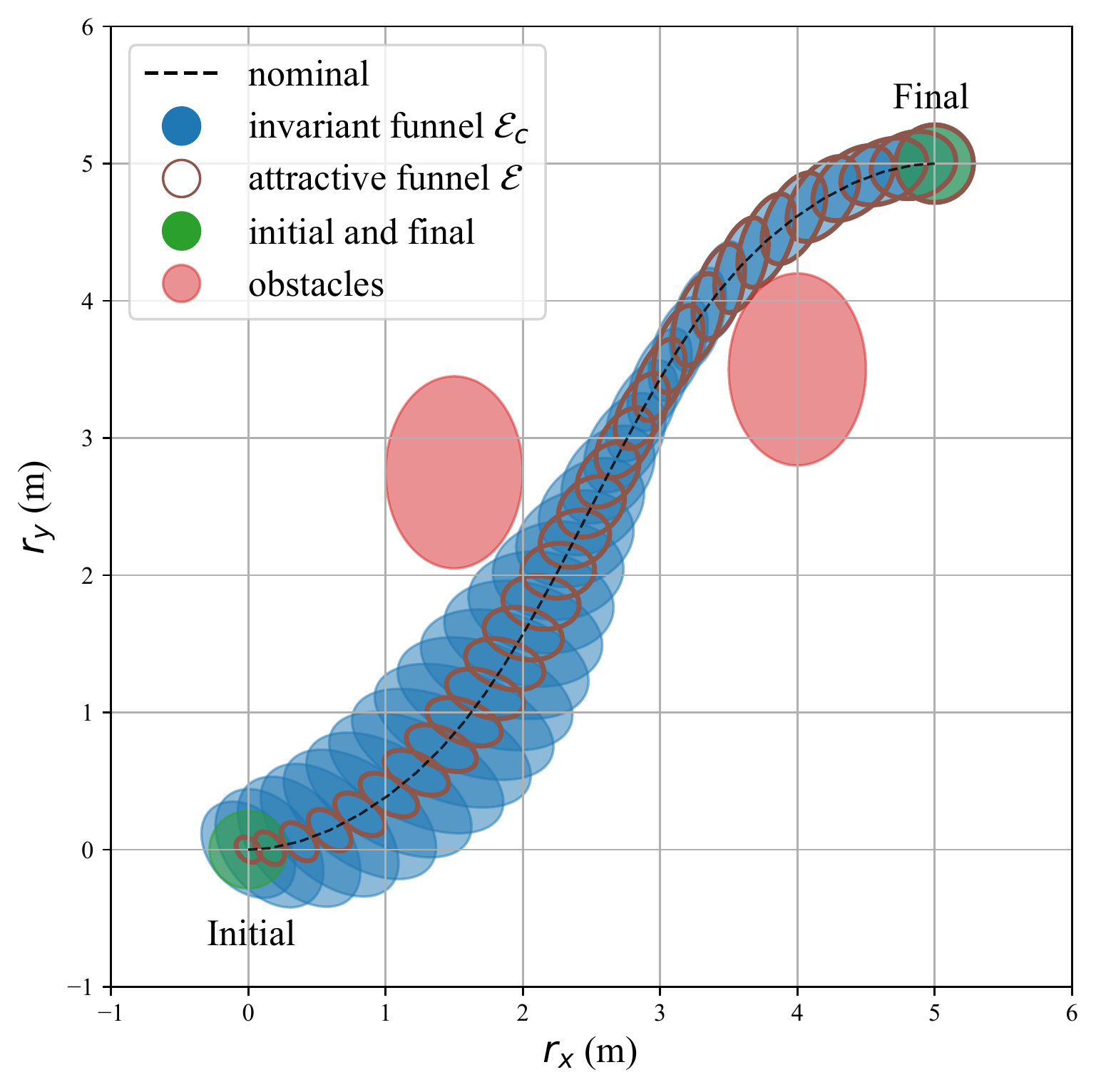} 
\caption{
The figure of the nominal trajectory and synthesized funnels projected on position coordinates. It shows the projection of $\mathcal{E}$ (brown ellipsoid) and $\mathcal{E}_c$ (blue ellipsoid) at each node point.} 
\label{fig:funnel_result}
\end{center}
\end{figure}

\begin{figure}
\begin{center}
\vspace*{0.05in}
\includegraphics[width=8.5cm]{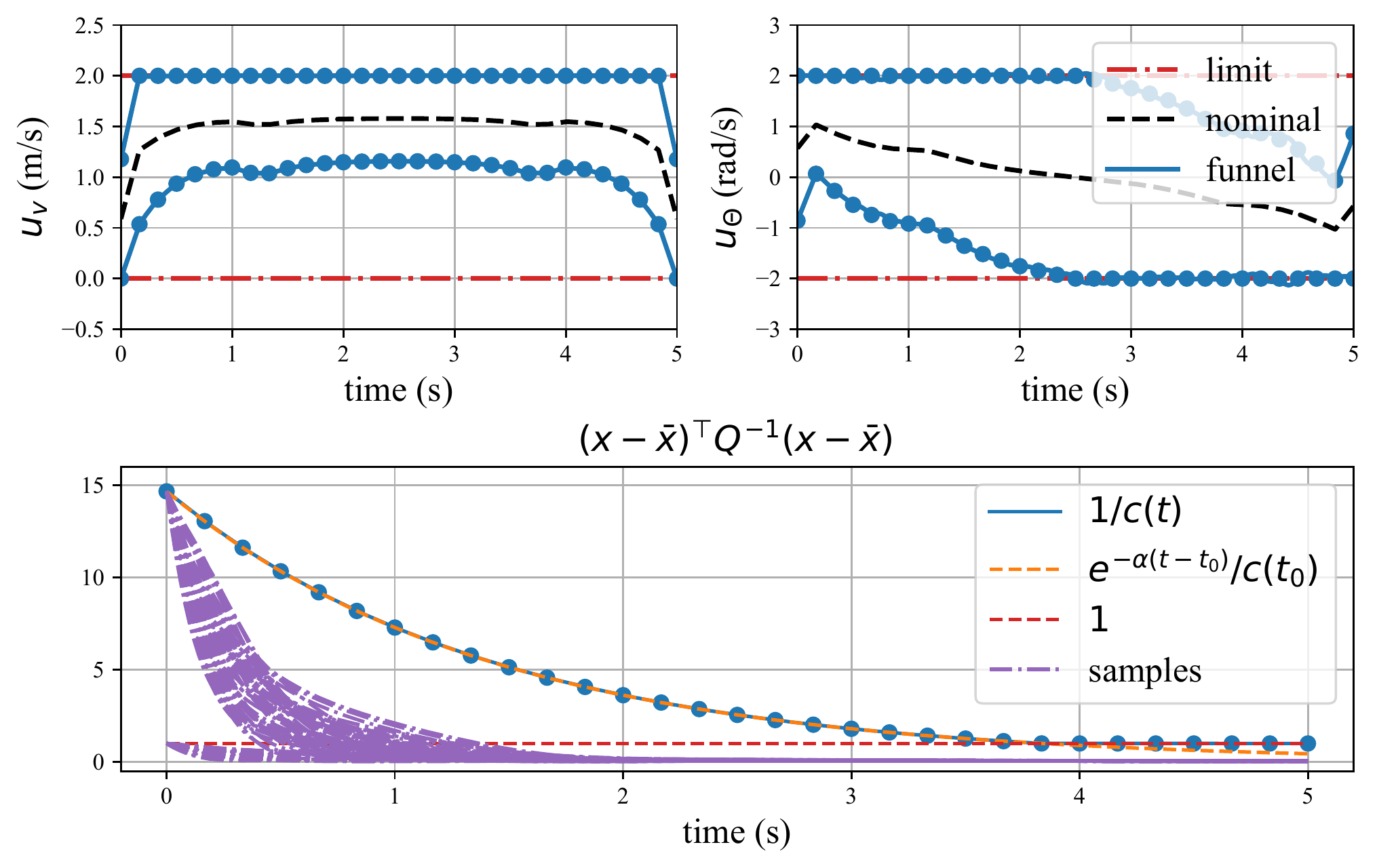} 
\caption{The figure of input funnel (top left and right) and support value $1/c(t)$ (bottom).} 
\label{fig:input_result}
\end{center}
\end{figure}

For the numerical simulation, we consider a unicycle model with addictive disturbances written as
\begin{align}
\left[\begin{array}{c}
\dot{r_x}\\
\dot{r_y}\\
\dot{\theta}
\end{array}\right] & =\left[\begin{array}{c}
u_{v}\cos\theta\\
u_{v}\sin\theta\\
u_{\theta}
\end{array}\right]+\left[\begin{array}{c}
0.1w_{1}\\
0.1w_{2}\\
0
\end{array}\right],\label{eq:unicycle_model}
\end{align}
where $r_x$, $r_y$, and $\theta$ are a $x$-axis position, a $y$-axis position, are a yaw angle, respectively, and $u_{v}$ is a velocity and $u_{\theta}$ is an angular velocity. The values $w_1$ and $w_2$ are disturbances such that $w=[w_1,w_2]^\top, \norm{w}\leq1$. In this model, the argument $q \in \mathbb{R}^2$ for the nonlinearity in \eqref{eq:Lure} is $[\theta, u_v]^\top$. We consider $N=30$ nodes evenly distributed over a time horizon of $5$ seconds with $t_{0}=0$ and $t_{f}=5$. The boundary parameters  $Q_i$ and $Q_f$ are both $\text{diag}([0.08~0.08~0.06])$. We consider two obstacle avoidance constraints, leading to nonconvex constraints for the state illustrated in Fig.~\ref{fig:funnel_result}. The input constraints are given as: $0\leq u_v \leq 2$ and $\abs{u_\theta} \leq 2$. The 100 samples around the nominal trajectory are used for the local Lipschitz constant $\gamma_k$ estimation for all $k$ in $\mathcal{N}_0^N$ by following the procedure given in \cite{reynolds2021funnel}. The weights $w_c$, $w_{Q_0}$, and $w_Q$ are $10^{3}$, $0.1$, and $0.1$, respectively. The parameters $\alpha$ and $\lambda_w$ are $0.7$ and $0.5$, respectively. The simulation can be reproducible by using the code at  \href{https://github.com/taewankim1/funnel_synthesis_multiple_shooting}{\footnotesize{\path{https://github.com/taewankim1/funnel_synthesis_multiple_shooting}}}.

The results of the proposed work are given in Fig.~\ref{fig:funnel_result} and Fig.~\ref{fig:input_result}. We can see that the generated funnel satisfies both state constraints (obstacle avoidance) and input constraints at each node point. Also, the resulting support value $1/c$ satisfies the constraint \eqref{eq:condition_c}. To test the invariance and attractivity conditions, we take a total of 100 samples, 50 from the surface of $\mathcal{E}(t_0)$ and 50 from that of $\mathcal{E}_c(t_0)$. We propagate each sample through the model \eqref{eq:unicycle_model} with a randomly chosen disturbance $w$ such that $\norm{w}=1$. In the bottom figure of Fig.~\ref{fig:funnel_result}, the value of the Lyapunov function for each sample trajectory is plotted. We can see that the invariance conditions of both $\mathcal{E}$ and $\mathcal{E}_c$ hold well, and the samples starting from the surface of $\mathcal{E}_c$ converge to the attractive funnel $\mathcal{E}$ due to the attractivity condition.

\section{Conclusions}\label{sec:conclusions}

This paper presents a funnel synthesis method for locally Lipschitz nonlinear systems under the presence of bounded disturbances. The proposed funnel synthesis approach aims to maximize the funnel entry while minimizing the attractive funnel to bound the effect of the disturbances. To solve the continuous-time funnel optimization problem having the DLMI, we apply the direct multiple shooting optimal control method. In the numerical evaluation with the unicycle model, the results show that the generated funnel satisfies both invariance and feasibility properties.


\bibliographystyle{IEEEtran}
\bibliography{root}

\end{document}